\documentclass[10pt]{amsart}

\usepackage[utf8x]{inputenc}
\usepackage[T1]{fontenc}
\usepackage[english]{babel}
\usepackage{amsmath,amssymb,amsfonts,amsthm}
\usepackage{amsmath,amsthm,amsfonts,amssymb,eucal,hyperref}
\usepackage[pdftex]{graphicx}
\usepackage{color}
\usepackage{enumerate}
\usepackage{fourier}

\newtheorem{thm}{Theorem}[section]
\newtheorem{lem}[thm]{Lemma}

\newtheorem{cor}[thm]{Corollary}

\theoremstyle{definition}
\newtheorem{defi}[thm]{Definition}
\newtheorem{rem}[thm]{Remark}

\newcommand{\de}{\, \mathrm{d}}

\newcommand{\Vol}{\operatorname{Vol}}

\newcommand{\N}{\mathbb{N}}
\newcommand{\Z}{\mathbb{Z}}

\newcommand{\R}{\mathbb{R}}

\newcommand{\T}{\mathbb{T}}

\newcommand{\CE}{\mathcal{E}}
\newcommand{\CF}{\mathcal{F}}
\newcommand{\CG}{\mathcal{G}}

\newcommand{\CK}{\mathcal{K}}
\newcommand{\CL}{\mathcal{L}}
\newcommand{\CM}{\mathcal{M}}

\newcommand{\CT}{\mathcal{T}}

\newcommand {\bx}{\mathbf x}

\newcommand {\bh}{\mathbf h}

\newcommand {\by}{\mathbf y}

\newcommand {\bn}{\mathbf n}
\newcommand {\bxi}{\boldsymbol \xi}

\newcommand{\beps}{\boldsymbol{\varepsilon}}

\newcommand{\abs}[1]{\left\lvert #1 \right\rvert}
\newcommand{\set}[1]{\left\{ #1 \right\}}

\newcommand{\bo}\boldsymbol{}
\newcommand{\bigo}[2][]{O_{#1}\left( #2 \right)}
\newcommand{\smallo}[2][]{o_{#1}\left( #2 \right)}
\newcommand{\bone}{\bo{1}}

\DeclareMathOperator{\spn}{span}

\DeclareMathOperator{\diag}{diag}
\DeclareMathOperator{\inj}{inj}

\newcommand{\ft}[2][]{\left[\CF_{#1} #2\right]}

\renewcommand{\tilde}{\widetilde}
\newcommand{\eps}{\varepsilon}
\renewcommand{\phi}{\varphi}

\newcommand{\GL}{\textrm{GL}}

\begin{document}

\title[Optimal tori and anisotropic expansion]{Eigenvalue optimisation on flat tori and lattice points in anisotropically expanding domains}
\author{Jean Lagac\'e}
\begin{addresses}
\address{Department of Mathematics \\ University College London \\ Gower Street, London \\ WC1E 6BT \\ United Kingdom}
\email{j.lagace@ucl.ac.uk}
\end{addresses}

 \begin{abstract}
 This paper is concerned with the maximisation of the $k$th eigenvalue of the Laplacian amongst flat tori of unit volume in dimension $d$ as $k$ goes to infinity. We show that in any dimension maximisers exist for any given $k$, but that any sequence of maximisers degenerates as $k$ goes to infinity when the dimension is at most $10$. Furthermore, we obtain specific upper and lower bounds for the injectivity radius of any sequence of maximisers. We also prove that flat Klein bottles maximising the $k$th eigenvalue of the Laplacian exhibit the same behaviour. These results contrast with those obtained recently by Gittins and Larson, stating that sequences of optimal cuboids for either Dirichlet or Neumann boundary conditions converge to the cube no matter the dimension. We obtain these results via Weyl asymptotics with explicit control of the remainder in terms of the injectivity radius. We reduce the problem at hand to counting lattice points inside anisotropically expanding domains, where we generalise methods of Yu. Kordyukov and A. Yakovlev by considering domains that expand at different rates in various directions.
\end{abstract}
\maketitle
\section{Introduction and main results}

Let $(M,g)$ be a smooth closed Riemannian manifold of dimension $d$. We study the Laplace eigenvalue problem
\begin{equation*}
 \Delta u + \lambda u = 0.
\end{equation*}
The eigenvalues of the Laplacian form a discrete, nondecreasing sequence, repeating every eigenvalue according to multiplicity, 
\begin{equation*}
0 = \lambda_0(M,g) \le \lambda_1(M,g) \le \dotso \nearrow \infty 
\end{equation*}
accumulating only at infinity. 

\subsection{Asymptotic eigenvalue optimisation}

In this paper, we study the maximisation problem
\begin{equation} \label{eq:statedproblem}
 \Lambda_k^\star(\CG) := \sup_{g \in \CG} \Lambda_k(M,g) := \sup_{g \in \CG}\Vol_g(M)^{2/d} \lambda_k(M,g),
\end{equation}
where $\CG$ is a class of metrics on $M$. This problem has been studied extensively for $k = 1$ in many settings: closed manifolds, manifolds with Neumann boundary conditions, and manifolds with Dirichlet boundary conditions in which case one minimises $ \Lambda_k$. Note that for closed manifolds it only makes sense to maximise $ \Lambda_k$. Indeed, for any $k$ one can find a sequence of metrics $g_n$ of unit volume such that $ \Lambda_k(M,g_n) \to 0$ as $n \to \infty$ by considering a sequence of metrics that degenerate to a disjoint union of $k+1$ closed manifolds touching at a point.

An interesting feature is that the extremisers for low eigenvalues are in general very symmetric. Indeed, the Faber-Krahn inequality \cite{faber,krahn1,krahn2} and the Szeg\"o-Weinberger inequality \cite{szego,weinberger} imply that the ball is the extremiser for $ \Lambda_1$ with Dirichlet or Neumann boundary conditions in any dimension. In the case of closed surfaces, Hersch has shown \cite{hersch} that the round sphere is the maximiser for $ \Lambda_1$ amongst two-dimensional spheres, and Nadirashvili has shown \cite{nad} that the equilateral flat torus is the maximiser for $\Lambda_1$ amongst surfaces of genus one.

For higher eigenvalues on domains, one does not expect those symmetries to appear. Indeed, A. Berger has shown \cite{aberger} that disks or union of disks can minimise $\Lambda_k$ on domains in the plane with Dirichlet boundary conditions only finitely many times. Furthermore, numerical experiments of Antunes and Freitas \cite{antunesfreitasnumerical} suggest that optimal domains in $\R^2$ may not exhibit many symmetries for $k \ge 5$. However, the same authors investigated in \cite{antunesfreitas} the behaviour of optimal domains as $k$ goes to infinity. More specifically, they showed that amongst rectangles with Dirichlet boundary condition, the sequence of rectangles minimising $\Lambda_k$ converges to the square in the Hausdorff metric. This has led to a series of papers \cite{vBBG,vBG,GL} culminating in a proof by Gittins and Larson, who show that in any dimension and with either Neumann or Dirichlet boundary conditions the sequence of optimal cuboids converges to the cube. 

Without any restriction on the metric, one does not even have a maximiser amongst closed manifolds. Indeed, Colbois and Dodziuk have shown in \cite{cd} that amongst all metrics of fixed volume on a manifold, one can make $\lambda_1$ as large as possible.
For metrics on closed surfaces, one does not necessarily expect the sequence of maximising metrics to converge to a smooth metric. For instance, Karpukhin, Nadirashvili, Penskoi and Polterovich \cite{KNPP} obtained in a recent preprint that the maximising metric on the two-dimensional sphere for the $k$th Laplace eigenvalue degenerates to a union of $k$ kissing round spheres. 

We study the maximisation problem \eqref{eq:statedproblem} for metrics on two classes of closed manifold. The first one is the class $\CM$ of flat metrics on tori in dimension $d$. Let $\CL = \GL_d(\R)/\GL_d(\Z)$ be the set of lattices in $\R^d$ equipped with the quotient topology. We identify $\CM$ with $\CL$ since
\begin{equation*}
 \CM = \set{\T_{\Gamma} = \R^d / \Gamma : \Gamma \in \CL}. 
\end{equation*}
As such, convergence in $\CM$ will be identified with convergence in $\CL$. We study the properties of maximisers to \eqref{eq:statedproblem} in $\CL_0$ the subset of all lattices with unit determinant, which corresponds to subset $\CM_0$ of flat tori with unit volume. 

The second class that we study is the set $\CE$ of flat metrics on Klein bottles. Flat Klein bottles are quotients of two-dimensional flat rectangular tori and as such are described by the two-parameters family
\begin{equation*}
 \CE := \set{K(a,b) := \left(\R^d/(a\Z \oplus b\Z)\right)/\sim : (a,b) \in \R^2_+},
\end{equation*}
where $\sim$ is the relation $(x,y) \sim \left(x + \frac a 2,b-y\right)$. Once again, we study the properties of maximisers of \eqref{eq:statedproblem} in the class $\CE_0$ of Klein bottles with unit volume, \emph{i.e.} the family $K(a,b)$ where $ab = 2$.

Before discussing asymptotic properties of maximisers to the problem \eqref{eq:statedproblem}, we start by proving that such maximisers do exist.
\begin{thm}\label{thm:existtorus}
 For all $k \in \N$, there exist $\T_{k}^\star \in \CM_0$ and $K_k^\star \in \CE_0$ maximising the variational problems
 \begin{equation*}
  \Lambda_k^\star(\CM) = \sup_{\T_{\Gamma} \in \CM}  \Lambda_k(\T_\Gamma).
 \end{equation*}
 and
 \begin{equation*}
   \Lambda_k^\star(\CE) = \sup_{K \in \CE}  \Lambda_k(K).
 \end{equation*}
\end{thm}

The behaviour of maximisers for tori and Klein bottles contrasts both with the results obtained for cuboids where the optimal cuboid converges to the cube and with the degeneracy results of \cite{cd} and \cite{KNPP}. Indeed, we show that for tori of dimension $2 \le d \le 10$, the sequence of optimisers has no limit points in $\CM_0$. However, we also show that this degeneracy can happen without changing the curvature as was done in \cite{KNPP}, or in \cite{cd}. 

Furthermore, we obtain a rate of degeneracy in terms of the injectivity radius. This is similar to the results in \cite{GL} where the rate of convergence to the cube is given. The range $2 \le d \le 10$ are the dimensions for which the volume of the unit ball $\omega_d$ is larger than $\omega_1 = 2$. In higher dimensions, the same type of result may hold, but the degeneracy certainly doesn't happen in the same way.
\begin{thm} \label{thm:optimtorus}
 In dimension $2 \le d \le 10$, there are no accumulation points in $\CM_0$ of any sequence $\set{\T_k^\star}$. The injectivity radius of $\T_k^\star$ respects
 \begin{equation} \label{eq:injradius}
 k^{-\frac{(1 - d)^2}{d}} \ll \operatorname{inj}(\T_k^\star) \ll k^{-\frac{1}{d}}.
 \end{equation}
 The lower bound is valid for all dimensions $d \in \N$. 
\end{thm}

\begin{rem}
 In dimension $2$, the lower bound and the upper bound are, at least to polynomial order, the same. The discrepancy between the upper and lower bounds are due to the fact that we find lower bounds on both the first and last successive minima of the associated dual lattice $\Gamma^*$, defined in equation \eqref{eq:succmin}. The lower bound on the last successive minima of $\Gamma^*$ gives directly an upper bound on the first successive minima of $\Gamma$ via Banaszczyk's transference theorem, and this quantity corresponds to the injectivity radius of $\T_\Gamma$. The lower bound on the first successive minima of $\Gamma^*$, does not give a lower bound on the injectivity radius so directly, and there is a loss in the strength of the estimation.
\end{rem}

In \cite{KLO}, Kao, Lai and Osting conjectured that in dimension $2$, the optimal flat torus was given by $\T_{2k}^\star = \R^2 / \Gamma_{2k}$, where $\Gamma_{2k}$ is the lattice spanned over $\Z$ by the vectors
 \begin{equation} \label{eq:optimalKLO}
  \gamma_1^{(2k)} = \left(k^2 - 1/4\right)^{-1/4}\left(1,0\right) \qquad \gamma_2^{(2k)} = \left(k^2 - 1/4\right)^{-1/4}\left(1/2,\sqrt{k^2 - 1/4} \right).
 \end{equation}
In dimension $2$, flat tori of unit volume form a two-dimensional moduli space with parameters $a,b$, with $a \in (-1/2,1/2]$, $b > 0$ such that $a^2 + b^2 \ge 1$. The associated lattices are spanned by
\[
 \gamma_1(a,b) = b^{-1/2}(1,0) \qquad \gamma_2(a,b) = b^{-1/2}(a,b).
\]
It is shown in \cite{KLO} that the flat torus in equation \eqref{eq:optimalKLO} is indeed maximal for $\Lambda_{2k}$ amongst tori for which $a^2 + b^2 \ge (k -1)^2$. The  upper bound on the injectivity radius in Theorem \ref{thm:optimtorus} yields that there exists a constant $c > 0$ such that the same torus has a higher $\Lambda_{2k}$ than every flat tori such that $a^2 + b^2 \le c k^2$.

Our methods also allow us to study sequences of optimisers in the moduli space $\CE$ of flat Klein bottles. Indeed, we also have degeneracy in this case, and we can also describe the rate of degeneracy.

\begin{thm} \label{thm:klein}
 There are no accumulation points in $\CE_0$ of any sequence $\set{K_k^\star}$. The injectivity radius of $K_k^\star$ respects
 \begin{equation} \label{eq:injradiusklein}
 k^{-\frac 1 2} \ll \operatorname{inj}(K_k^\star) \ll k^{-\frac 1 2}.
 \end{equation}
\end{thm}

\subsection{Explicit exponent for the remainder in Weyl's law}

In the papers \cite{antunesfreitas,vBBG,vBG,GL} on optimal cuboids a prominent feature consisted in finding uniform bounds on the eigenvalue counting function
\begin{equation*}
N(\lambda;M) = \#\set{\lambda_k(M) < \lambda}.
\end{equation*}
 Weyl's law states that for any fixed $(M,g)$ the counting function $N(\lambda;M)$ enjoys the asymptotics
\begin{equation}\label{eq:weylremainder}
 N(\lambda;M) = \frac{\omega_d}{(2\pi)^d} \lambda^{\frac d 2} + R(\lambda;M),
\end{equation}
where $R(\lambda;M) = \smallo{\lambda^{\frac{d}{2}}}$ and $\omega_d$ is the volume of a unit ball in dimension $d$. Under the hypothesis that periodic geodesics have measure $0$ in the cosphere bundle of $M$, Duistermaat and Guillemin \cite{DuistermaatGuillemin} have shown that the remainder in equation \eqref{eq:weylremainder} satisfies
\begin{equation} \label{eq:remaindernotsharp}
 R(\lambda;M) = \smallo{\lambda^{\frac{d-1}2}}.
\end{equation}
Note that the size of $R(\lambda;M)$ depends on the geometry of $M$ in a non trivial way. Indeed, for any fixed $\lambda$ one can find a sequence $g_n$ of metrics on $M$ such that $N(\lambda;(M,g_n)) \to \infty$ as $n \to \infty$ for the same reason one can make $\lambda_k$ arbitrarily small.  However, one can still ask under what geometric conditions on $M$ does there exists a function $R(\lambda)$ such that
\begin{equation} \label{eq:unifremainder}
 N(\lambda;M) = \frac{\omega_d}{(2\pi)^d} \lambda^{\frac d 2} + R(\lambda)
\end{equation}
with $R(\lambda)= \bigo{\lambda^{\tau}}$ independent of $M$, with $\tau < d/2$. The search for this type of uniform bounds was a prominent feature in the above mentioned papers \cite{antunesfreitas,vBBG,vBG,GL}. The presence of the boundary allowed them to derive a two-term Weyl type bound; closed manifolds do not exhibit this behaviour. 

In \cite[Theorem 6.2]{buserisoperimetric}, Buser has obtained bounds on the eigenvalue $\lambda_k$ of a closed manifold, valid when $k$ was large enough in terms of the injectivity radius, see also \cite[equation 1.2.5]{HKP} where this result is reformulated in terms of the counting function. The following theorem  states that we can find explicit bounds on the remainder in \eqref{eq:unifremainder} depending on the injectivity radius.

\begin{thm}\label{thm:weyltorus}
  There is $C > 0$ such that for all $\lambda \ge 2\pi$ and all flat tori of unit volume we have that
\begin{equation} \label{eq:weyltorusexponent}
 \abs{N(\lambda;\T_{\Gamma}) - \frac{\omega_d}{(2\pi)^d} \lambda^{d/2}} \le C \lambda^{\frac{d}{2} - \frac{d}{d+1}} \inj(T_{\Gamma})^{- \frac{2d}{d+1}}.
\end{equation}
Moreover, for any flat Klein bottle $K(a,b) \in \CE_0$
\begin{equation*}
 \abs{N(\lambda;K(a,b)) - \frac{1}{4\pi} \lambda} \le C \lambda^{\frac 1 3} \inj(K(a,b))^{- \frac 2 3}.
\end{equation*}
\end{thm}

We make the following remarks as to the naturality and sharpness of those results.

\begin{rem}
 The remainder in the previous theorem is natural in the following sense. If we take the normalisation
 \[
  N_0(\lambda) = \lambda^{- \frac d 2}N(\lambda) 
 \]
 such that $N_0(\lambda)$ has a limit as $\lambda \to \infty$, then the remainder obtained in Theorem \ref{thm:weyltorus} is invariant under homothetic rescaling of the metric.
 \end{rem}

\begin{rem}
 If $\inj(\T_\Gamma)$ is of the order of $\lambda^{-1/2}$, the remainder in Theorem \ref{thm:weyltorus} is of the order of the principal term. This can indeed happen: as part of the proof of Theorem \ref{thm:optimtorus} we will construct an explicit sequence of flat tori $\T_k \in \CM_0$ such that
 \[
  \inj(\T_k) = \frac{\lambda_{2k}(\T_k)^{-1/2}}{2\pi}
 \]
 whose eigenvalue counting functions satisfy
 \[
    \abs{N(\lambda_{2k}(\T_k);\T_k) - \frac{\omega_d}{(2 \pi)^d}\lambda_{2k}(\T_k)^{d/2}} \gg \lambda^{d/2}.
 \]
In fact, one will be able to compute explicitly
 \[
  \abs{N(\lambda_{2k}(\T_k);\T_k) - \frac{\omega_1}{(2 \pi)^d}\lambda_{2k}(\T_k)^{d/2}} = 2 d - 1, 
 \]
and $\omega_1 \ne \omega_d$.

This also implies that one cannot improve the order of error term in the spectral parameter without making it worse in terms of the injectivity radius, and vice versa.
\end{rem}

\subsection{Lattice points inside domains} \label{sec:geonumb}

We translate the problems at hand in the language of lattice point counting. The spectrum of the Laplacian on a flat torus is given by
\begin{equation} \label{eq:spectrumtorus}
 \sigma(\T_\Gamma) = \set{4 \pi^2 \abs{\gamma^*}^2: \gamma^* \in \Gamma^*},
\end{equation}
where $\Gamma^*$ is the lattice dual to $\Gamma$ defined by
\begin{equation*}
 \Gamma^* := \set{\gamma^* \in \R^d : (\gamma^*, \Gamma) \subset \Z}.
\end{equation*}
Similarly, the spectrum of the Laplacian on a flat Klein bottle is giveni in \cite{bergergauduchonmazet} to be
\begin{equation} \label{eq:spectklein}
 \sigma(K(a,b)) := \set{4 \pi^2 \left(\frac{m^2}{a^2} + \frac{n^2}{b^2}\right) : (m,n) \in \Z \times \N_0, (m,n) \ne (2\ell+1,0)}.
\end{equation}

A classical problem in the geometry of numbers consists in counting the number of points of an isotropically shrinking lattice $\Gamma_{\lambda}:= \lambda^{-1}\Gamma$ inside a domain $\Omega$ containing the origin as $\lambda \to \infty$. This dates back to the Gauss circle problem and has been studied in great details for various type of domains over the years. Denote 
\begin{equation*}
 \abs{\Omega} = \Vol_d(\Omega) \quad \text{and} \quad \abs{\Gamma} = \det(A_\Gamma),
\end{equation*}
where $A_\Gamma$ is any matrix such that $A_\Gamma \Z^d = \Gamma$. In general, one aims for asymptotics of the form
\begin{equation} \label{eq:weyl}
 N(\Omega;\Gamma_\lambda) := \#\left(\Omega \cap \Gamma_\lambda\right) = \frac{|\Omega|}{\abs{\Gamma_{\lambda}}} + R(\lambda;\Omega;\Gamma), 
\end{equation}
where
\begin{equation} \label{eq:discr}
 R(\lambda;\Omega;\Gamma) = \bigo{\abs{\Gamma_\lambda}^{-\eta}}
\end{equation}
with $\eta < 1$ The implicit constant on the righthand side of equation \eqref{eq:discr} depends on the geometry of $\Omega$, the geometry of its boundary, and on $\Gamma$. In general, given non compact families of lattices or domains, the implicit constant is not uniform and therefore the formula \eqref{eq:weyl} cannot be used directly to find extremisers to $N(\Omega;\Gamma_\lambda)$ for large $\lambda$. Note that maximising this counting function does not makes sense, even while keeping the lattice determinant and the volume of the domain fixed. Indeed, for a fixed $\Omega$ containing the origin and $\eps$ small enough the lattice $\eps^{d-1} \Z \oplus \eps^{-1} \Z^{d-1}$ has arbitrarily many points in $\Omega$ and determinant $1$. 

We formulate the results of the two previous sections in terms of lattices. From the fact that
\begin{equation*}
 \#\set{\Z^d \cap A_\Gamma^{-1}(B_1)} = \#\set{A_\Gamma\Z^d \cap B_1},
\end{equation*}
the following two questions are equivalent.
\begin{itemize}
\item What's the largest lattice determinant of a lattice with at least $k$ points in $B_1$?
\item What's the smallest area of an ellipsoid enclosing at least $k$ points of the lattice $\Z^d$?
\end{itemize}
Symmetry of ellipsoids or lattices with respect to the transformation $x \mapsto -x$ means that no generality is lost by asking these questions for only even (or odd) $k$. Let us order elements of any lattice as
\begin{equation*}
 \Gamma = \set{\gamma_k : k \in \N_0}
\end{equation*}
with $\gamma_0 = 0$ and $\gamma < \tilde \gamma$ if $|\gamma| < |\tilde \gamma|$, and if their norms are equal by lexicographic order. The scaling invariance of the problem is made explicit by studying maximisers to the functional
\begin{equation*}
 \tilde \Lambda_k(\Gamma) = \abs{\Gamma}^{-1/d}\abs{\gamma_k}.
\end{equation*}
We obtain the following restatement of Theorem \ref{thm:existtorus} in terms of lattices.
\begin{thm}[Lattice version of Theorem \ref{thm:existtorus}] \label{thm:exist}
 For every $k \in \N$, there exists $\Gamma_k^\star \in \CL$ maximising $\tilde \Lambda_k$. 
\end{thm}

\begin{rem}
 The maximiser in the previous theorem is not unique, in particular if $\Gamma$ is a maximiser, then $\mu \Gamma$ is also one. We will, depending on what is pertinent at the right moment, either normalise them by determinant or by $\abs{\gamma_k}$. Note that even within $\CL_0$ unicity is not guaranteed.
\end{rem}
We now study properties of the maximisers $\Gamma_k^\star$. The degeneracy of a sequence $\Gamma_k^\star$ is given in terms of their \emph{successive minima}, the lattice invariants $\mu_j(\Gamma)$ defined for $1 \le j \le d$ by
\begin{equation} \label{eq:succmin}
 \mu_j(\Gamma) = \inf\set{\mu : \dim(\spn(\Gamma \cap B_{\mu})) \ge j}.
\end{equation}
We prove the following restatement of Theorem \ref{thm:optimtorus}.
\begin{thm}[Lattice version of Theorem \ref{thm:optimtorus}] \label{thm:optim}
 Let $\set{\Gamma_k^\star}\subset \CL_0$ be a sequence of maximisers of $\tilde\Lambda_k$ normalised by $\abs{\Gamma_k^\star} = 1$, in dimension $d \le 10$.  Then, the following holds.
 
 \begin{enumerate}
 
 \item The sequence $\Gamma_k^\star$ has no accumulation points in $\CL_0$. 
 \item The successive minima of the sequence $\Gamma_k^\star$ satisfy the asymptotic bounds
 \begin{equation*}
  \mu_1\left(\Gamma_k^\star\right) \gg k^{- 1 + \frac 1 d}
 \end{equation*}
 and
 \begin{equation*}
  \mu_d\left(\Gamma_k^\star\right) \gg k^{\frac 1 d}.
 \end{equation*}
 \end{enumerate}
\end{thm}
This will be proved thanks to the following restatement of Theorem \ref{thm:weyltorus} in terms of lattices.

\begin{thm}[Lattice version of Theorem \ref{thm:weyltorus}] \label{thm:weyllattice}
There exists a constant $C$ such that for all lattices with $\abs{\Gamma} \le 1$ 
\begin{equation} \label{eq:boundmu1}
 \abs{N(B_1;\Gamma) - \frac{\omega_d}{\abs{\Gamma}}} \le C \abs{\Gamma}^{- 1} \mu_d(\Gamma)^{\frac{2d}{d+1}} .
\end{equation}
\end{thm}

\subsection{Plan of the paper and sketch of the proofs}

We start in Section \ref{sec:general} by exposing general facts about lattices that will be used in the sequel. More specifically, we describe the relevant lattice invariants and state theorems of Minkowski and Banaszczyk that are important later, for ease of reference.

In Section \ref{sec:optimisation}, we prove Theorems \ref{thm:exist} and \ref{thm:optim}. Inspired by a construction of Kao, Lai and Osting \cite{KLO} in dimension $2$, we produce in Section \ref{sec:optimalcandidate} in any dimension a sequence of lattices $\Theta_{2k}$ such that
\begin{equation} \label{eq:counterexample}
 \abs{\theta_{2k-1}} = \abs{\theta_{2k}} =\left(\frac{2k
 }{\omega_1}\right)^{1/d}.
\end{equation}
However,  Theorem \ref{thm:weyllattice} implies that for any lattice $\Gamma$ of unit determinant whose successive minima satisfy $\mu_d(\Gamma) = \smallo{k^{1/d}}$, we have that
\begin{equation*}
 \abs{\gamma_{2k-1}} = \abs{\gamma_{2k}} \le   \left(\frac{2k}{\omega_d}\right)^{1/d}(1 + \smallo 1)
\end{equation*}
with $\omega_d$ the volume of the unit ball. 
One can see that while the sequence $\omega_d$ converges to $0$ as $d \to \infty$, it is initially increasing. Indeed, for all $2 \le d \le 10$, we have that $\omega_d > \omega_1$. 

In Section \ref{sec:equivalence}, we will show that the spectral theoretic versions of Theorems \ref{thm:existtorus}, \ref{thm:optimtorus} and \ref{thm:weyltorus} are implied by Theorems \ref{thm:exist}, \ref{thm:optim} and \ref{thm:weyllattice} using Banaszczyk's transference theorem \ref{thm:banaszczyk} and Minkowski's successive minima theorem \ref{thm:minkowski}. 

In Section \ref{sec:anisotropic}, we switch gears and describe Theorem \ref{thm:weyllattice} in terms of points of $\Z^d$ sitting inside anisotropically expanding domains. These were studied by Yu. Kordyukov and A. Yakovlev in a series of papers \cite{KY2011,KY2012,KY20152,KY2015} and we generalise their results and methods to our setting. 

In Section \ref{sec:poisson}, we prove the theorems about the number of points of a lattice sitting inside anisotropically expanding domains using the Poisson summation formula method. In the classical version of this problem, one uses global estimates on the Fourier transform of the indicator of a convex set to obtain bounds on the counting function of lattice points inside an expanding domain. It is, however, not possible to make this kind of computations uniformly when the expansion is anisotropic. The main idea, inspired by \cite{KY20152} is to only use Fourier transform estimates along the subspace where the expansion is the fastest and to use trivial $L^\infty$ estimates in the orthogonal complement.

\subsection*{Acknowledgements}

This work is part of the author's doctoral studies at Université de Montréal under the supervision of Iosif Polterovich. We thank Pedro Freitas, Katie Gittins, Corentin Léna and Braxton Osting for useful discussions. The research of the author was supported by NSERC's Alexander-Graham-Bell doctoral scholarship.

\section{Some facts about lattices in \texorpdfstring{$\R^d$}{Rd}} \label{sec:general}

For most standard results on lattices, one can see \cite{Cassels}. The set of all full-rank lattices in $\R^d$ can be realised as $\CL = \GL_d(\R)/\GL_d(\Z)$, equipped with the quotient topology. A lattice $\Gamma \in \CL$ is identified with its generator matrix $A_\Gamma$, the matrix such that $A_\Gamma \Z^d = \Gamma$. Every lattice determines uniquely a flat torus $\T_\Gamma = \R^d/\Gamma$.

Two relevant lattice invariants that are of interest in this paper are the determinant (or volume) and the successive minima. The determinant is defined as
\begin{equation*}
 |\Gamma| := \det A_\Gamma = \Vol_d(\T_\Gamma).
\end{equation*}
By convention, we assign to the trivial lattice a volume of $1$. The successive minima $\mu_j(\Gamma)$ are defined for $1 \le j \le d$ as 
\begin{equation*}
 \mu_j(\Gamma) := \inf\set{\mu : \dim(\spn(\Gamma \cap B_{\mu})) \ge j}.
\end{equation*}
Note that $\mu_j$ is always attained, \emph{i.e.} there is always $\gamma \in \Gamma$ such that $\mu_j(\Gamma) = \abs{\gamma}$. Furthermore, the first successive minimum gives the injectivity radius of the associated torus, \emph{i.e.}
\begin{equation*}
 \mu_1(\Gamma) = \operatorname{inj}(\T_\Gamma).
\end{equation*}

The successive minima of a lattice and the determinant are related through a theorem of Minkowski.

\begin{thm}[Minkowski's sucessive minima theorem] \label{thm:minkowski}
 Let $\mu_1,\dotsc,\mu_d$ be the successive minima of a lattice $\Gamma$. Then, there exists constants $c,C>0$ such that
 \begin{equation*}
  c \abs{\Gamma} \le \prod_{j=1}^d \mu_j \le C \abs \Gamma.
 \end{equation*}
\end{thm}

To any lattice $\Gamma$ we associate the dual lattice
\begin{equation*}
 \Gamma^* = \set{\gamma^* \in \R^d : (\gamma^*,\Gamma) \subset \Z}.
\end{equation*}
The operation $*$ is a continuous involution on $\CL$; hence a set $\CK \subset \CL$ is compact if and only if $\CK^*$ is. Let $A_{\Gamma}$ be the generating matrix for $\Gamma$, then $A_{\Gamma^*} = (A_{\Gamma}^*)^{-1}$; from this we infer that $\abs{\Gamma^*} = \abs{\Gamma}^{-1}$.

The following theorem from Banaszczyk \cite{banaszczyk} is also useful in the sequel and relates the successive minima of $\Gamma$ and those of $\Gamma^*$.
\begin{thm}[Banaszczyk's transference theorem] \label{thm:banaszczyk}
 For any $1 \le j \le d$, the following inequalities hold between the successive minimas of the lattices $\Gamma$ and $\Gamma^*$ :
 \begin{equation*}
  1 \le \mu_j(\Gamma) \mu_{d - j + 1}(\Gamma^*) \le d.
 \end{equation*} 
\end{thm}

The lattice invariants can be used to characterise compactness in $\CL$, by Mahler's selction theorem \cite{Cassels}[Theorems 5.3, 5.4 and Lemma 8.3]. This theorem states that a set $\CK \subset \CL$ is compact if and only if the determinant is bounded and the first minimum $\mu_1$ is bounded away from zero on $\CK$. Equivalently, it is compact if and only if the determinant is bounded away from zero and $\mu_d$ is bounded on $\CK$. Compactness in the moduli space of all flat tori is obtained by identifying a torus with its lattice.

\begin{defi}
 A sequence of lattices $\set{\Gamma_k}$ is said to \emph{degenerate} if either $\abs{\Gamma_k} \to \infty$ or if $\mu_1(\Gamma_k)  \to 0$. In other words, it degenerates if it is not contained in some compact set in $\CL$.
\end{defi}

We will be interested in the number of lattice points inside the unit ball $B_1$, denote this quantity $N(\Gamma;1)$. Denoting by $\bone_S$ the indicator function of a set $S$, we have that
\begin{equation*}
 N(\Gamma;1) = \sum_{\gamma \in \Gamma} \bone_{B_1}(\gamma).
\end{equation*}

%
%
%
%
%
%

%
%
Finally, we say that a subspace $V \subset \R^d$ is a $\Gamma$-subspace if it is spanned by a subset of $\Gamma$. The set $\Gamma(V):= \Gamma\cap V$ is a lattice in $V$. 

%
%
%
%
%

\section{Optimal lattices}\label{sec:optimisation}

In this section, we prove Theorems \ref{thm:exist} and \ref{thm:optim} assuming Theorem \ref{thm:weyllattice}. Order elements of a lattice $\Gamma$ with respect to their norms and by lexicographic order whenever the norms are equal. We write $\Gamma = \set{\gamma_k : k \in \N_0}$. We study sequences of lattices maximising the functionals 
\begin{equation*}
\tilde \Lambda_k(\Gamma) = |\Gamma|^{-1/d}|\gamma_k|.
\end{equation*}
Note that for any lattice $\Gamma$ and $m \ge 1$ we have that $\tilde \Lambda_{2m-1}(\Gamma) = \tilde \Lambda_{2m}(\Gamma)$; we will therefore only consider maximisers for even $k$. 

\subsection{Proof of Theorem \ref{thm:exist}}

 Consider a maximising sequence $\set{\Gamma_n}$ for $\tilde \Lambda_k$. Without loss of generality from the definition of $\tilde \Lambda_k$ we may suppose that $|\Gamma_n| = 1$ for all $n$. Suppose that $\mu_1(\Gamma_n) \to 0$. Then, for some $n$ we have that $\mu_1(\Gamma_n) < 1/k$. Let $\gamma \in \Gamma_n$ be a lattice point realising $\mu_1(\Gamma_n)$. Then, $1 > |k \gamma| > \abs{\gamma_k}$. However, the $k$th element of $\Z^d$ has norm greater than $1$, contradicting that $\set{\Gamma_n}$ was a maximising sequence. By Mahler's selection theorem, $\set{\Gamma_n}$ has a convergent subsequence, and by continuity of the norm and the determinant, it converges to a maximiser for $\tilde \Lambda_k$. 
 
\qed

\subsection{Lattices with large \texorpdfstring{$\tilde \Lambda_k$}{Lambda k}} \label{sec:optimalcandidate}

In this section we study a specific sequence of lattices that we will use as a measuring stick for other sequences of lattices. Note that we make no claim of these lattices being the optimisers. Consider the lattices
\begin{equation*}
 \Theta_{2k} = k^{-1 + \frac 1 d} \Z \oplus k^{\frac 1 d} \Z^{d-1}.
\end{equation*}
Then, we have 
\begin{equation*}
\abs{\theta_{2k-1}} = \abs{\theta_{2k}} = k^{1/d}
\end{equation*}
and
\begin{equation*}
 \abs{\Theta_{2k}} = 1.
\end{equation*}
In particular, we have that
\begin{equation*}
 \tilde \Lambda_{2k}(\Theta_{2k}) = k^{1/d}
\end{equation*}
which will be the quantity to beat. Observe that the sequence $\Theta_{2k}$ degenerates and that
\begin{equation*}
 \mu_d(\Theta_{2k}) = k^{1/d}.
\end{equation*}

\subsection{Proof of Theorem \ref{thm:optim}}

Let $\set{\Gamma_k}$ be a sequence of lattices of unit volume such that $\mu_d(\Gamma_k) = \smallo{k^{1/d}}$. We will show that under such conditions, $\Gamma_k$ cannot be a maximiser for $\tilde \Lambda_k$ infinitely often. This is done by showing that for large $k$ and any fixed $t>0$,
\begin{equation*}
\#(B_{k^{1/d} - t} \cap \Gamma_{2k}) > 2k, 
\end{equation*}
implying that 
\begin{equation*}
\tilde \Lambda_{2k}(\Gamma_{2k}) \le k^{1/d} - t < \tilde \Lambda_{2k}(\Theta_{2k}).
\end{equation*}

We have that
\begin{equation*}
\#\left(B_{k^{1/d} - t} \cap \Gamma_{2k}\right) = \#\left(B_{1} \cap \frac{1}{k^{1/d} - t} \Gamma_{2k}\right),
\end{equation*}
that
\begin{equation*}
 \mu_d\left(\frac{1}{k^{1/d} - t}\Gamma_{2k}\right) = \smallo{1},
 \end{equation*}
 and that
 \begin{equation*}
 \abs{\frac{1}{k^{1/d} - t} \Gamma_{2k}} = k^{-1}\left(1 - tk^{- 1 /d}\right)^{-d}.
\end{equation*}
We therefore satisfy the hypotheses of Theorem \ref{thm:weyllattice} and therefore get
\begin{equation*}
 \#\left(B_{k^{1/d} - t} \cap \Gamma_{2k}\right) = \omega_d (1 - tk^{-1/d})^d k(1 + \smallo 1)
\end{equation*}
For $2 \le d \le 10$, we have that $\omega_d > \omega_1 = 2$. Hence, there is $K$ such that for $k> K$  
\begin{equation*}
 \#(B_{k^{1/d} - t} \cap \Gamma_{2k}) > 2k,
\end{equation*}
proving that there is a finite number of maximisers in the sequence $\set{\Gamma_k}$. This implies that there is  constant $c$ such that any sequence of normalised maximisers respects $\mu_d(\Gamma_k) \ge c k^{1/d}$, also implying that the sequence degenerates.

For the lower bound on $\mu_1(\Gamma_k)$, any sequence $\Gamma_k$ normalised by determinants such that $\mu_1(\Gamma_{2k}) < k^{-1 + 1/d}$ has that
\begin{equation*}
 \tilde \Lambda_{2k}(\Gamma_{2k}) \le k \mu_1(\Gamma_{2k}) < \tilde \Lambda_{2k}(\Theta_{2k}),
\end{equation*}
hence this is not a sequence of maximisers. 

\qed

\section{From lattices to tori} \label{sec:equivalence}

In this section we prove the spectral theoretic versions of Theorems \ref{thm:existtorus}, \ref{thm:optimtorus} and \ref{thm:weyltorus}, as well as Theorem \ref{thm:klein}. For any lattice $\Gamma$ we denote by $\gamma^*_k$ the $k$th ordered element of the dual lattice $\Gamma^*$. Since $\lambda_k(\T_\Gamma) = 4 \pi^2\abs{\gamma_k^*}^2$ and $\Vol(\T_{\Gamma}) = \abs{\Gamma^*}^{-1}$, we have that
\begin{equation*}
 \Lambda_k(\T_\Gamma) = \left(2\pi\tilde \Lambda_k(\Gamma^*)\right)^2.
\end{equation*}
Since these quantities are positive the problem of maximising $\Lambda_k$ on flat tori is the same as the problem of maximising $\tilde \Lambda_k$ on the dual lattices of those tori.
\subsection{Proof of Theorem \ref{thm:existtorus}}

By Theorem \ref{thm:exist} there exists a lattice $\Gamma_k^\star$ maximising $\tilde \Lambda_k$. The torus with lattice $\Gamma = (\Gamma_k^\star)^*$ is therefore a maximiser for $\Lambda_k$. 

For flat Klein bottles, we have from equation \eqref{eq:spectklein} that the eigenvalues of $K(a,b)$ are continuous in the parameters $a$ and $b$. Normalising by $ab = 2$, it is easy to see that for any $k$, $\lambda_k(K(a,b))$ goes to $0$ when either $a$ or $b$ goes to zero. Hence for any fixed $k$ we can restrict ourselves to a compact subset of the parameters $a,b$ and the maximiser exists.

\qed

\subsection{Proof of Theorem \ref{thm:optimtorus}}

Denote by $\Gamma_k^\star$ a sequence of optimal lattices with unit determinant for $\Lambda_{k}$ and denote by $\T_k^\star$ the corresponding optimal torus $T_k^\star = \R^d/(\Gamma_k^\star)^*$. Since compactness of a set $\CK \subset \CL_0$ is equivalent to compactness of the set of duals $\CK^*$, we have that the sequence of optimal tori degenerates.

We now turn to the geometric constraints. Recall that $\operatorname{inj}(\T_k^\star) = \mu_1((\Gamma_k^\star)^*)$. By Banaszczyk's transference theorem, we have that
\begin{equation*}
 \mu_1((\Gamma_k^\star)^*) \le \frac{d}{\mu_d(\Gamma_k^\star)}.
\end{equation*}
Hence, from the lower bound for $\mu_d(\Gamma_k^\star)$ in Theorem \ref{thm:optim} we have that
\begin{equation*}
 \operatorname{inj}(\T_k^\star) = \mu_1((\Gamma_k^\star)^*) \ll k^{- 1/d}.
\end{equation*}
On the other hand, by Minkowski's successive minima theorem, there is a constant $C$ such that 
\begin{equation*}
\begin{aligned}
 \mu_d(\Gamma_k^\star) &\le C \mu_1(\Gamma_k^\star)^{1 - d}\\
 &\le k^{\frac{(1 - d)^2}{d}}.
 \end{aligned}
\end{equation*}
Once again, Banaszczyk's transference theorem yields
\begin{equation*}
 \operatorname{inj}(T_{k}^\star) \ge k^{- \frac{(1 - d)^2}{d}},
\end{equation*}
finishing the proof.

\qed

\subsection{Proof of Theorem \ref{thm:klein}}

For flat Klein bottles, observe that the injectivity radius of $K(a,b)$ is given by
\begin{equation*}
 \inj(K(a,b)) = \min(a,b/2).
\end{equation*}
Let $\Gamma(a,b)$ be the lattice defined by 
\begin{equation*}
 \Gamma(a,b) := \frac{2\pi}{a} \Z \oplus \frac{2 \pi}{b} \Z
\end{equation*}
It is not hard to see that $\Gamma(a,b)$ has the property
\begin{equation*}
 N(\lambda;K(a,b)) = \frac{1}{2} \#\left(\Gamma(a,b) \cap B_{\sqrt \lambda}\right) + \bigo 1.
\end{equation*}
Indeed, let $\Xi(a,b)$ be the set
\begin{equation*}
 \Xi(a,b) := \left(\frac{2 \pi}{a}\Z \oplus \frac{2 \pi}{b}N_0\right) \setminus \set{\frac{2 \pi}{a}(2 \ell + 1, 0): \ell \in \Z}.  
\end{equation*}
Then, the spectrum of $K(a,b)$ is the same as the square of the norm of elements of $\Xi(a,b)$. However, it is easy to see that if we take the union of $\Xi(a,b)$ and $-\Xi(a,b)$, we recover $\Gamma(a,b)$ except for points of the form $\left(\frac{2(2\ell + 1) \pi}{a},0\right)$, but we added twice the elements of the form $\left(\frac{4 \pi \ell)}{a}, 0\right)$. Hence, we have that
\begin{equation*}
\abs{ \#(\Gamma(a,b) \cap B_{\sqrt{\lambda}}) - \#\left(\Xi(a,b) \cap B_{\sqrt\lambda}\right) - \#\left(- \Xi(a,b) \cap B_{\sqrt{\lambda}}\right)} \le 3. 
\end{equation*}
Now, for rectangular lattices we have that $\mu_1(\Gamma(a,b) = 2 \pi \min(a^{-1},b^{-1})$ and $\mu_2(\Gamma(a,b)) = 2 \pi \max(a^{-1},b^{-1})$. The rest of the analysis is performed exactly in the same way as for flat tori.

\qed

\subsection{Proof of Theorem \ref{thm:weyltorus}}

 Let $\T_\Gamma$ be any flat torus of unit volume. Observe that, by Banaszczyk's transference theorem we have that
\[
 \inj(\T_\Gamma) \asymp \mu_d(\Gamma^*)^{-1}.
\]
We have from equation \eqref{eq:spectrumtorus} that
\[
 N(\lambda;T_\Gamma) = \#\left(2 \pi \lambda^{-1/2}\Gamma^* \cap B_1 \right)
\]
Denote by $\Gamma^*_\lambda$ the rescaled lattice $2 \pi \lambda^{-1/2} \Gamma^*$. By  Theorem \ref{thm:weyllattice}, we have that
\begin{equation} \label{eq:weyltorusfromlattice}
\begin{aligned}
 N(\lambda;T_\Gamma) = \frac{\omega_d}{\abs{\Gamma_\lambda^*}} + \bigo{\abs{\Gamma_\lambda^*}^{ - 1} \mu_d(\Gamma_\lambda)^{\frac{2d}{d+1}}}
 \end{aligned}
\end{equation}
We have that
\[
 \mu_d(\Gamma_\lambda^*) = 2 \pi\lambda^{-1/2}  \mu_d(\Gamma^*) 
\]
and that

\[
 \abs{\Gamma_\lambda^*} = \frac{(2 \pi)^d}{\lambda^{d/2}}.
\]
Inserting those values into equation \eqref{eq:weyltorusfromlattice} yields the desired asymptotic in Theorem \ref{thm:weyltorus}. 

\qed

\section{Anisotropically expanding domains} \label{sec:anisotropic}

We now ground the statement of Theorem \ref{thm:weyllattice} in terms of the counting of lattice points sitting inside anisotropically expanding domains developped by Yu. Kordyukov and A. Yakovlev in \cite{KY2011,KY2012,KY20152,KY2015}. Consider the decomposition of $\R^d$ as
\begin{equation*}
 \R^d := E := \bigoplus_{j=1}^d V_j.
\end{equation*}
We will use $E$ to refer to a specific decomposition for $\R^d$. For $\beps = (\eps_1,\dotsc,\eps_d)$ consider the linear transformation $T_{\beps}$ given by
\begin{equation*}
 T_{\beps} = \sum_{j = 1 }^d \eps_j^{-1} \bx_j.
\end{equation*}
with $\bx_j \in V_j$. Without loss of generality we suppose that $\eps_1 \le \dotso \le \eps_d$.  We denote the set of all such transformations $\CT_E$, and by $\CT$ the union of all such transformations over decompositions $E$, and we say that $T_{\beps}$ is anisotropic whenever not all $\eps_j$ are equal.

For $\Omega$ a bounded subset of Euclidean space and $\Gamma \in \CL$, denote
\begin{equation*}
 n_{\beps}(\Omega;\Gamma;\by) := \#\left(\Gamma \cap \left(T_{\beps}\Omega + \by\right) \right) = \#\left(T_{\beps}^{-1}(\Gamma - \by) \cap \Omega\right).
\end{equation*}

Kordyukov and Yakovlev have studied asymptotics for $n_{\beps}$ in the specific case where a subspace $V$ of $\R^d$ is fixed, and $\Omega$ is stretched along its orthogonal complement. In our notation, this corresponds to $E = V_1 \oplus V_2$ with $\eps_1 \to 0$ and $\eps_2 = 1$. 

In our case, the expansion is happening at different rates along different subspaces. We split the remainder of this section in three parts. First, we describe asymptotics for $n_{\beps}$ in terms of the decomposition $E$ with an explicit dependence on the $\eps_j$. Then, we show that from the perspective of the counting function, we can describe any lattice using the transformations $T_{\beps}$. Finally, we derive Theorem \ref{thm:weyllattice} from Theorem \ref{thm:recursive}.

\subsection{Lattice points inside anisotropically expanding domains}

We start by fixing some notation. Denote $\abs \beps = \det T_{\beps} = \prod_j \eps_j$; asymptotic results will be given in terms of $\abs \beps$ as it goes to zero, and in terms of how fast the $\eps_j$'s goes to zero in relation to $\abs{\beps}$. Let us split the decomposition $E$ into three parts. Let us first write
\begin{equation*}
V_0 = \bigoplus_{j : \eps_j = 0} V_j,
\end{equation*}
and let $W$ be the maximal $\Gamma^*$-subspace in $V_0$, and write $d_W = \operatorname{dim}(W)$ and $\Gamma_W = \Gamma^*(V)^*$. We further decompose $E$ as
\begin{equation*}
 E = V \oplus V' \oplus W
\end{equation*}
in such a way that $\Gamma^* \cap V' = \set 0$. We set $d_V = \operatorname{dim}(V)$ and $d_{V'}$ analogously. Finally, denote 
\begin{equation*}
 \delta_V = \|T_{\beps}^{-1}\|_V
\end{equation*}
the norm of $\T_{\beps}^{-1}$ restricted to $V$.\footnote{If $V = V_{j_1} \oplus \dotso \oplus V_{j_m}$ with the indices in increasing order, this is equal to $\eps_{j_m}$}
We obtain the following theorem.

\begin{thm} \label{thm:recursive}
 Suppose that $\Omega$ is a bounded open subset of $\R^d$ with smooth boundary such that for all $\gamma \in \Gamma_W$, $\Omega \cap (\gamma + W^\perp)$ is strictly convex. Then,
\begin{equation} \label{eq:mainthm}
 n_{\beps}(\Omega;\Gamma;\by) = \frac{\abs\beps^{-1}\abs{\Gamma_W}}{\abs{\Gamma}} \sum_{\gamma \in \Gamma_W} \Vol(\Omega \cap (\gamma + W^\perp)) + \bigo{\abs\beps^{-1}\delta_V^{\frac{2 d_V}{1 + d_V + 2 d_{V'}}}},
\end{equation}
with the implicit constant only dependant on $\Gamma_W$, $V'$ and \/ $\Omega$.
\end{thm}

\begin{rem}
 If $W = \Gamma_W = \set 0$, the condition on $\Omega$ becomes strict convexity, and the asymptotic formula becomes 
\begin{equation*}
  n_{\beps}(\Omega;\Gamma;\by) = \frac{|\Omega|}{|\Gamma|} \abs\beps^{-1} + \bigo{\abs\beps^{-1}\delta_V^{\frac{2 d_V}{1 + d_V + 2 d_{V'}}}}
 \end{equation*}
with the implicit constants dependant on $V'$ and $\Omega$.

Furthermore, from \cite{KY2015}[Section 3.2], this is the only case we need to prove. Indeed, they show that there are $\by_\gamma$ such that
\begin{equation*}
 n_{\beps}(\Omega;\Gamma;\by) = \sum_{\gamma \in \Gamma_W} n_{\beps}(\Omega \cap (\gamma + W^\perp);\Gamma(W^\perp);\by_\gamma).
\end{equation*}
They then show in \cite{KY2015}[Lemma 3.3] that $\Gamma(W^\perp)^* \cap (V' \cap W^\perp) = \set 0$. Since the sum in equation \eqref{eq:mainthm} is finite, we obtain the desired result by applying Theorem \ref{thm:recursive} with $W = \set 0$ term by term.
 \end{rem}

\subsection{From \texorpdfstring{$\CT$}{T} to lattices}

We start by showing that we can restrict ourselves to lattices of the form $T_{\beps}^{-1} \Z^d$ in our investigation of   Theorem \ref{thm:weyllattice}. 

\begin{lem} \label{prop:equiveps}
 For every $\Gamma \in \CL$, there exists a decomposition 
 \[
 \R^d = E = \bigoplus_{j=1}^d V_j
 \]
and $T_{\beps} \in \CT_E$ such that
 \begin{equation} \label{eq:counting}
  N(\Gamma;B_1) = n_{\beps}\left(B_1;\Z^d;0\right). 
 \end{equation}
 For every $T_{\beps} \in \CT$, there exists $\Gamma$, such that equation \eqref{eq:counting} holds.
\end{lem}

\begin{proof}
 Let $A_\Gamma \in \GL_d(\R)$ be such that $A_\Gamma \Z^d = \Gamma$. Then,
 \begin{equation*}
  \begin{aligned}
   \sum_{\gamma \in \Gamma} \bone_{B_1}\left(\gamma\right) &= \sum_{\gamma \in \Gamma} \bone_{A_\Gamma(B_1)}\left(A_\Gamma^{-1} \gamma\right) \\
   &= \sum_{\bn \in \Z^d} \bone_{A_\Gamma( B_1)}\left( \bn \right) 
  \end{aligned}
 \end{equation*}
Observe now that since $B_1 = \set{\bx \in \R^d : \bx^* \bx \le 1 }$, then 
\[
 A_\Gamma(B_1) = \set{\bx \in \R^d : \bx^* (A_\Gamma^*)^{-1} A_\Gamma^{-1} \bx \le 1}.
\]
 Since $(A_\Gamma A_\Gamma^*)^{-1}$ is symmetric definite positive, it can be diagonalised as 
\begin{equation*}
 (A_\Gamma A_\Gamma^*)^{-1} = U^* D^{1/2} D^{1/2} U
\end{equation*}
 with $U$ orthogonal.  Let $\beps = \diag(D^{1/2})$ and $V_j$ be eigenspaces of $(A_\Gamma A_\Gamma^*)^{-1}$. Since $N(\Gamma;B_1)$ is invariant under orthogonal transformations of $\Gamma$, we have that
 \begin{equation*}
  N(\Gamma;1) = N(U\Gamma;1) = \#\set{A_{\Gamma}^{-1} U \Gamma \cap A_\Gamma^{-1} B_1} = n_{\beps}(B_1).
 \end{equation*}
On the other hand, this process can be inverted : given $T_{\beps}$, we take $\Gamma$ to be the lattice with generating matrix $T_{\beps}^{-1}$. 

\end{proof}

The previous lemma allows us to consider only the lattices of the form $\Gamma = T_{\beps}^{-1} \Z^d$. The following lemma relates the lattice invariants to the associated transformation $T_{\beps}$.

\begin{lem} \label{lem:succminbounds}
 Let $\Gamma$ be a lattice in $\CL$. Then, for any $T_{\beps} \in \CT$ such that $\Gamma = T_{\beps}^{-1} U \Z^d$ for some orthogonal transformation $U$ we have that 
 \begin{equation*}
  \abs{\Gamma} = \det\left(T_{\beps}^{-1}\right) = \abs{\beps}.
 \end{equation*}
and that the following bounds hold for the successive minima $\mu_1(\Gamma)$ and $\mu_d(\Gamma)$ :
\begin{equation*}
  \eps_1 \le \mu_1(\Gamma) \le \mu_d(\Gamma) \le \eps_d.
\end{equation*}
Furthermore, one can choose $T_{\beps}$ such that $\Gamma = T_{\beps}^{-1} U \Z^d$ and
\begin{equation*}
  \mu_1(\Gamma) \le \frac{d^{5/2}}{2} \eps_1\qquad \text{and} \qquad \mu_d(\Gamma) \ge  \frac{2}{d^{3/2}}\eps_d.
\end{equation*}
\end{lem}

\begin{proof}
 Without loss of generality, since the determinant and successive minima are invariant under orthogonal transformations we suppose that $U = I$. The assertion on determinants holds by definition. Let $\bn$ be any non-zero element of $\Z^d$, and write $\bn = \bn_1 + \dotso + \bn_d$ with $\bn_j \in V_j$. Then,
 \begin{equation*}
 \begin{aligned}
  \abs{T_{\beps}^{-1}\bn}^2 &= \sum_{j=1}^d \eps_j^{2} \abs{\bn_j}^{2} \\
  &\ge \eps_1^2 \sum_{j=1}^d \abs{\bn_j}^2 \\
  &= \eps_1^2 \abs{\bn}^2.
  \end{aligned}
 \end{equation*}
 Since $\bn \ne 0$, we have that $\mu_1(\Gamma) \ge \eps_1$.
For the upper bound on $\mu_d$, observe that any $T_{\beps}$ sends bases of $\R^d$ to bases of $\R^d$. As such, from the definition of $\mu_d$ we have that
\begin{equation*}
\begin{aligned}
 \mu_d(\Gamma) &\le \sup_{\bn \in \Z^d \setminus \set 0} \frac{\abs{T_{\beps}^{-1}\bn}}{\abs{\bn}}, \\
 &\le \eps_d.
\end{aligned}
 \end{equation*}
 
 We now obtain the lower bound on $\mu_d$ for a specific $T_{\beps}$. There is a basis of $\Gamma$ whose elements all have norm smaller than $\frac{d \mu_d(\Gamma)}{2}$ \cite[Lemma V.8]{Cassels}. Let $T_{\beps,\Gamma}^{-1}$ be the square root of the diagonalised Gram matrix $G_{\Gamma}$ associated to that basis. By Cauchy-Schwartz, the entries of the Gram matrices $G_\Gamma$ all bounded by $\frac{d^2 \mu_d(\Gamma)^2}{4}$. Let $\nu_d(G_\Gamma)$ be the largest eigenvalue of $G_{\Gamma}$. It satisfies the bound
  \[
  \nu_d\left(G_\Gamma\right) \le \sqrt{\operatorname{tr}(G^*_\Gamma G_\Gamma)} \le \frac{d^3 \mu_d(\Gamma)^2}{4}. 
 \]
 Note that the eigenvalues of $G_\Gamma$ are the same as those of $T_{\beps,\Gamma}^{-2}$, hence we have that
 \begin{equation*}
  \eps_d \le \frac{d^{3/2}}{2} \mu_d(\Gamma),
 \end{equation*}
yielding the desired result. For the upper bound on $\mu_1$, observe that a generating matrix for $\Gamma^*$ is $T_{\beps}$. Hence, by the previous argument we have that
\begin{equation*}
 \mu_d(\Gamma^*) \ge \frac{2}{d^{3/2}} \eps_1^{-1}.
\end{equation*}
From Banaszczyk's transference theorem, we can then infer that
\begin{equation*}
 \mu_1(\Gamma) \le d \mu_d(\Gamma^*)^{-1} \le \frac{d^{5/2}}{2} \eps_1,
\end{equation*}
finishing the proof.
 
\end{proof}

\subsection{Proof of Theorem \ref{thm:weyllattice}}

Given lattice $\Gamma$ with $\abs \Gamma < 1$, we know from Lemma \ref{prop:equiveps}, one can find a decomposition $E$ of $\R^d$ and a transformation $T_{\beps} \in \CT_E$ such that 
\begin{equation*}
 N(\Gamma;B_1) = n_{\beps}(B_1;\Z^d;0).
\end{equation*}
Furthermore, from Lemma \ref{lem:succminbounds} we get that one can choose $T_{\beps}$ in such a way that
\begin{equation*}
 \eps_d \le \frac{d^{3/2}}{2} \mu_d(\Gamma).
\end{equation*}
We therefore satisfy the hypotheses of Proposition \ref{thm:recursive} with $V = E$, $V' = W = \set{0}$, and $\delta_E = \eps_d$ and we deduce that
\begin{equation*}
 N(\Gamma;B_1) = \omega_d \abs\beps^{-1} + \bigo{\abs \beps^{-1}\eps_d^{\frac{2d}{1 +d}}}.
\end{equation*}
Plugging in $\abs{\Gamma} = \abs \eps$ and $\mu_d \asymp \eps_d$ gives the desired asymptotics. 

\qed

\section{Asymptotic estimates} \label{sec:poisson}

In this section, we prove Theorem \ref{thm:recursive} using the Poisson summation formula.  We follow the structure set out by the author and Parnovski in \cite{lagaceparnovski}. The first thing we have to do is a mollification of $\bone_{\Omega}$ so that it is smooth enough for the Poisson summation formula to be used, and we will get estimates from above and below using the mollified functions. In the second part, we obtain estimates on partial Fourier transforms of such functions. Finally, we use the Poisson summation formula to obtain asymptotics for the counting function.

\subsection{Mollification}

Let $\rho \in C_c^\infty(\R^d)$ be a non-negative bump function supported in the unit ball and such that
\begin{equation*}
 \int_{\R^d} \rho(\bx) \de \bx = 1.
\end{equation*}
We also let $\bh = (h_1,\dotsc,h_d)$ be a set of parameters to be fixed later, and we set
\begin{equation*} \label{eq:mollifier}
 \rho_{\bh}(\bx) = \frac{1}{h_1\cdots h_d}\rho(T_\bh(\bx)).
\end{equation*}
Note that $\rho_\bh$ is supported in the ellipsoid
\begin{equation*}
 E_\bh = \set{\bx \in V : \left\|T_\bh \bx \right\|< 1}.
\end{equation*}
For any function $f : \R^d \to \R$ let $f^{(\bh)}$ be the mollification of $f$ by $\rho_\bh$, that is
\begin{equation*}
 f^{(\bh)}(\bx) = \left[f * \rho_{\bh}\right](\bx) = \int_{\R^d} f(\bx - \by) \rho_\bh(\by) \de \by.
\end{equation*}

Let us now approximate $\bone_\Omega$ by smooth functions. For any set $B$ define the sets
\begin{equation*}
 B_\bh = \bigcup_{\bx \in B} \left(\bx + E_\bh\right) \qquad \text{and} \qquad B_{-\bh} = \R^d \setminus \left(\R^d \setminus B\right)_\bh.
\end{equation*}
The following lemma will be needed about these sets.
\begin{lem} \label{lem:linearity}
 Let $\beps \bh = (\eps_1 h_1,\dotsc,\eps_d h_d)$ and $B \subset V$. Then,
 \begin{equation*}
  T_{\beps}(B)_{\pm\bh} = T_{\beps}(B_{\pm\beps \bh}).
 \end{equation*}

\end{lem}

\begin{proof}
 It follows simply from linearity of $T_{\beps}$ and the fact that $T_{\beps} E_\bh = E_{\beps \bh}$.
\end{proof}

We now prove that $\bone^{(h)}$ provides a good approximation to $\bone$.
\begin{lem}
 Let $\Omega \subset \R^d$ and $\bx \in \R^d$. Then,
 \begin{equation}\label{eq:ineq}
 \bone^{(\bh)}_{T_{\beps}(\Omega)_{- \bh}}\left(\bx\right) \le \bone_{T_{\beps}(\Omega)}\left(\bx \right) \le \bone^{(\bh)}_{T_{\beps}(\Omega)_{ \bh}}\left(\bx \right).
 \end{equation}
\end{lem}
\begin{proof}
 For any set $B$ we have that 
 \begin{equation*}
  0 \le \bone_B^{(\bh)} \le 1.
 \end{equation*}
Hence, to show the right most inequality in \eqref{eq:ineq} it suffices to show that for any $\bx \in T_{\beps}(\Omega) $ we have that  $\bone_{T_{\beps}(\Omega)_{\bh}}^{(\bh)}\left(\bx\right) = 1$. By definition $\bx + E_{\bh} \subset T_{\beps}(\Omega)_{ \bh}$ hence
\begin{equation*}
 \begin{aligned}
  \bone_{(T_{\beps}\Omega)_{\bh}}^{(\bh)}\left(\bx\right)&= \int_{E_{ \bh}} \rho_{\bh} (\by) \de \by = 1.
 \end{aligned}
\end{equation*}

To prove the left-most inequality in \eqref{eq:ineq}, it suffices to show that for any $\bx \in E \setminus T_{\beps}(\Omega)$  we have that  $\bone_{T_{\beps}(\Omega)_{-\bh}}^{(\bh)}(\bx) = 0$. We have that
\begin{equation*}
\bx + E_{\bh} \subset (\R^d \setminus T_{\beps}(\Omega))_\bh,
\end{equation*}
and $\bone_{(T_{\beps}(\Omega))_{-\bh}}$  is supported in the complement of that set. Hence,
\begin{align*}
\bone_{T_{\beps}(\Omega)_{\bh})}^{(\bh)}(\bx) &= \int_{E_\bh} \bone_{T_{\beps}(\Omega)_{-\bh}}(\bx-\by) \rho_\bh(\by) \de \by \\ 
&= 0,
\end{align*}
finishing the proof.
\end{proof}

The following corollary follows directly from the previous lemma.

\begin{cor}
Defining 
\begin{equation*}
 n_{\beps}^{\pm}(\Omega) = \sum_{\gamma \in \Gamma} \bone_{T_{\beps}(\Omega)_{\pm\bh}}^{(\bh)}(\gamma),
\end{equation*}
the inequalities
\begin{equation*}
 n_{\beps}^-(\Omega) \le n_{\beps}(\Omega) \le n_{\beps}^+(\Omega)
\end{equation*}
hold for all $\beps$.
\end{cor}

\subsection{Fourier transform estimates}

Let $V$ be a subspace of $\R^d$ and write $\bx = \bx_V + \bx'$ for any $\bx \in \R^d$. We define the $V$-Fourier transform of a sufficiently rapidly decaying function $f$ as
\begin{equation*} 
\begin{aligned}
 \ft[V]{f}(\bxi_V,\bx') &= \int_{V} e^{- 2 \pi i \bx_V \cdot \bxi_V} f(\bx_V,\bx') \de \bx_V.
 \end{aligned}
\end{equation*}
When $V = \R^d$, we will write $\ft f := \ft[\R^d]f$. We obtain estimates for the decay of $\ft f(\bx)$ in terms of $\ft[V]{f}$.
Observe that\begin{equation}\label{eq:wfour}\begin{aligned}
 \abs{\ft f(\bxi)} &= \abs{\int_{\R^d} e^{- 2 \pi i \bx \cdot \bxi} f(\bx) \de \bx },\\
 &= \abs{\int_{V^\perp} e^{- 2 \pi i \bx' \cdot \bxi'}\int_{V} e^{- 2 \pi i \bx_V \cdot \bxi_V} f(\bx) \de \bx_V \de \bx' },\\
 &= \abs{\int_{V^\perp} e^{- 2 \pi i \bx' \cdot \bxi'} \ft[V]f(\bxi_V,\bx') \de \bx'}, \\
 &\le \int_{V^\perp} \abs{\ft[V]f(\bxi_V,\bx')} \de \bx'.
\end{aligned}\end{equation}

From this we get the following lemma.

\begin{lem} \label{lem:parftconvex}
 Let $\Omega$ be a bounded domain, and $V$ be a subspace of dimension $d_V$ such that the intersection $\Omega \cap V$ is strictly convex. Then,
 \begin{equation*}
  \ft{\bone_\Omega}(\bxi) = \bigo{\abs{\bxi_V}^{-\frac{d_V+1}{2}}}.
 \end{equation*}
\end{lem}
\begin{proof}
 Standard results about the Fourier transform of the indicator of a strictly convex set (see e.g. \cite[Theorem 2.29]{IL}) tell us that
\begin{equation*}
 \ft[V]{\bone_\Omega}(\bxi_V,\bx') = \bigo[\bx']{\abs{\bxi_V}^{-\frac{d_V+1}{2}}}.
\end{equation*}
From equation \eqref{eq:wfour}, we have that
\begin{equation*}
\abs{\ft \bone_\Omega(\bxi)} \le  \int_{V^\perp} \abs{\ft[V]\bone_\Omega(\bxi_V,\bx')} \de \bx'.
\end{equation*}
Since $\ft[V]{\bone_\Omega}(\bxi_V,\bx')$ is compactly supported in $\bx'$, we obtain the desired result, finishing the proof.
\end{proof}

\subsection{Poisson summation formula}

Let us apply the Poisson summation formula to the smoothed sums $n_{\beps}^\pm(\Omega;\Gamma;\by)$. Denote $\Gamma' = \Gamma^* \setminus \set 0$ to obtain
\begin{equation} \label{eq:poissonbone}
\begin{aligned}
  n_{\beps}^\pm(\Omega,\by;0) &= \sum_{\gamma \in \Gamma} \bone_{T_{\beps}(\Omega)_{\pm \bh} + \by}^{(\bh)}(\gamma) = \frac{1}{|\Gamma|}\sum_{\gamma^* \in \Gamma^*} \ft{\bone_{T_{\beps}(\Omega)_{\pm \bh} + \by}^{(\bh)}}(\gamma^*); \\
  &= \frac{1}{\abs \Gamma} \ft{\bone_{T_{\beps}(\Omega)_{\pm \bh} + \by}^{(\bh)}}(0) + \Sigma(\beps,\bh,\by)
  \end{aligned}
\end{equation}
Observe that
\begin{equation} \label{eq:ftsplit}
  \ft{\bone_{T_{\beps}(\Omega)_{\pm \bh} + \by}^{(\bh)}}(\bxi) =  e^{i \by \cdot \bxi} \ft{\bone_{T_{\beps}(\Omega)_{\pm \bh} + \by}^{(\bh)}}(\bxi) \ft{\rho_\bh}(\bxi - \by).
\end{equation}
Since we will only find bounds using the absolute values of the terms in the previous equation, and since $$\ft{\bone_{T_{\beps}(\Omega)_{\pm \bh} + \by}^{(\bh)}}(0) = \ft{\bone_{T_{\beps}(\Omega)_{\pm \bh}}^{(\bh)}}(0)$$ we suppose without loss of generality that $\by = 0$.

We first turn our attention to $\ft{\bone_{T_{\beps}(\Omega)_{\pm \bh}}^{(\bh)}}(0)$. Using properties of the Fourier transform, and Lemma \ref{lem:linearity} we have that
\begin{equation*}
 \ft{\bone_{T_{\beps}(\Omega)_{\pm \bh}}}(\xi) = \abs \beps^{-1} \ft{\bone_{\Omega_{\beps \bh}}}(T_{\beps}(\bxi))
 \end{equation*}
 and from equation \eqref{eq:mollifier} that
 \begin{equation*} \label{eq:ftrho}
  \ft{\rho_{\bh}}(\bxi) = \ft{\rho} \left(T_{\bh}^{-1}(\bxi)\right).
\end{equation*}
Hence, the first term in equation \eqref{eq:poissonbone} is given by
\begin{equation*}
\begin{aligned}
 \ft{\bone_{T_{\beps}(\Omega)_{\pm \bh}}^{(\bh)}}(0)&= \frac{\abs \beps^{-1}}{\abs \Gamma}  \ft{\bone_{\Omega_{\beps \bh}}}(T_{\beps}(0)) \ft{\rho} \left(T_{\bh}^{-1}(0)\right) \\
  &=  \frac{\abs \beps^{-1}}{|\Gamma|} \Vol(\Omega_{\pm\beps \bh})
 \end{aligned}
\end{equation*}
As long as all the $\eps_j h_j$ remain bounded, we have that there exists a constant $C$ such that
\begin{equation*}
 \Vol( \Omega_{ \beps \bh} \setminus \Omega) \le C \left(\sum_{j=1}^d \eps_j h_j\right),
\end{equation*}
and
\begin{equation*}
 \Vol( \Omega \setminus \Omega_{- \beps \bh}) \le C \left(\sum_{j=1}^d \eps_j h_j\right),
\end{equation*}
hence, writing $\Omega_{\beps \bh} = \Omega \cup (\Omega_{\beps \bh} \setminus \Omega)$ and $\Omega = \Omega_{-\beps\bh} \cup (\Omega \setminus \Omega_{-\beps\bh})$ we have
\begin{equation*}
\begin{aligned}
\ft{\bone_{T_{\beps}(\Omega)_{\pm \bh}}^{(\bh)}}(0) &= \frac{\abs \beps^{-1}}{|\Gamma|} \Vol(\Omega) +  \bigo{\abs \beps^{-1}\sum_{k=1}^d \eps_k h_k}
\end{aligned}
\end{equation*}
Let us now study $\Sigma(\beps,\bh,0)$ in equation \eqref{eq:poissonbone}. Using equations \eqref{eq:ftsplit} and \eqref{eq:ftrho}  we deduce that
\begin{equation*}
\Sigma(\beps,\bh,0) = \abs\beps^{-1}\sum_{\gamma^* \in \Gamma'} \ft{\bone_{\Omega_{\beps \bh}}}\left(T_{\beps}(\gamma^*)\right) \ft{\rho}\left(T_{\bh}^{-1}(\gamma^*)\right).
\end{equation*}
We have that $\ft{\rho_{\bh}}$ is a Schwartz function, \emph{i.e.} for any $N$
\begin{equation*}
 \ft{\rho_\bh}(\bxi) = \bigo{(1 +|\bxi|)^{- N}},
\end{equation*}
hence we have that
\begin{equation*}
 \Sigma(\beps,\bh,0) = \abs\beps^{-1}\sum_{\gamma^* \in \Gamma'}\frac{\ft{\bone_{\Omega_{\beps \bh}}}\left(T_{\beps}(\gamma^*)\right)}{\left(1 + (h_1 \bx_1)^N + \dotso + (h_d \bx_d)^N\right)}.
\end{equation*}

\subsection{Proof of Propositon \ref{thm:recursive}} All that remains is to bound $\Sigma(\beps,\bh,0)$ and to balance it with the error term coming from the Fourier transform evaluated at $0$. Let $j_V = \operatorname{argmax}(j: V_j \subset V)$ and set $h_V = h_{j_V}$.
Choose
\begin{equation*}
 h_k = \delta_V^{\frac{2d_V}{1 +d_V + 2d_{V'}}} \eps_k^{-1},
\end{equation*}
hence
\begin{equation*}
\abs \bh = \delta_V^{\frac{2d_V (d_V + d_{V'})}{1 +d_V + 2d_{V'}}} \abs \beps^{-1}. 
\end{equation*}
 For all $\gamma^* \in \Gamma'$ we have that $\gamma^*_V \ne 0$. From Lemma \ref{lem:parftconvex}, we obtain the bound

\begin{equation} \label{eq:asympheps}
\begin{aligned}
\Sigma(\beps,\bh,0) &\ll \abs \beps^{-1} \sum_{\gamma^* \in \Gamma'} \frac{\delta_V^{\frac{d_V + 1}{2}}}{|\gamma^*_V|^{\frac{m + 1}{2}}\left(1 + (h_1 \gamma_1^*)^N + \dotso + (h_d \gamma_d^*)^N\right)} \\
 & \ll \delta_V^{\frac{d_V+1}{2}} \abs \beps^{-1} \int_{\R^d} \frac{|\bx_V|^{- \frac{d_V + 1}{2}}}{\left(1 + (h_1 \bx_1)^N + \dotso + (h_d \bx_d)^N\right)} \de \bx  \\
 & \ll (\delta_V h_V)^{-\frac{d_V+1}{2}}\abs \beps^{-1} \abs \bh^{-1}.
\end{aligned}
\end{equation}
Combining with the estimate on $\ft{\bone_{T_{\beps}(\Omega)_{\pm \bh}}^{(\bh)}}(0)$, we have that
\begin{equation} \label{eq:twosums}
\begin{aligned}
 n_{\beps}^\pm(\Omega) = \frac{\abs\beps^{-1}}{|\Gamma|} \Vol(\Omega) &+ \bigo{\abs \beps^{-1}\sum_{k = 1}^d \eps_k h_k} \\ &\qquad+  \bigo{(\delta_V h_V)^{-\frac{d_V+1}{2}} \abs \beps^{-1} \abs \bh^{-1}}.
 \end{aligned}
\end{equation}
Using the fact that $\abs \beps \le \delta_V^{d_V}$, we obtain that equation \eqref{eq:twosums} reduces to
\begin{equation*}
 n_{\beps}^\pm(\Omega) = \frac{\abs\beps^{-1}}{|\Gamma|} \Vol(\Omega) + \bigo{\abs \beps^{-1} \delta_V^{\frac{2d_V}{1 + d_V + 2 d_{V'}}}}.
\end{equation*}

\qed

\bibliographystyle{plain}
\bibliography{../../../biblio/biblio}

\end{document}